\documentclass[12pt]{amsart}
\usepackage{amsmath}

\def\p{\partial}

\newcommand{\esssup}{\mathop{\mathrm{ess\,sup}}} 
\newcommand{\essinf}{\mathop{\mathrm{ess\,inf}}}

\def\ttD{{\Delta}}
\def\ttN{{\nabla}}

\def\tg{{g}_{ij}}
\def\R{R_\abb}

\def\be{\begin{equation}}
\def\ee{\end{equation}}

\def\p{\partial}

\def\p{\partial}

\def\p{\partial}

\def\R{\Bbb R}

\newtheorem{thm}{Theorem}[section]

\newtheorem{lem}{Lemma}[section]
\newtheorem{prop}{Proposition}[section]
\newtheorem{cor}{Corollary}[section]
\theoremstyle{definition}

\theoremstyle{remark}

\newtheorem{rem}{Remark}[section]
\numberwithin{equation}{section}

\begin{document}

\title[Lagrangian mean curvature flow]
{Lagrangian Mean Curvature flow for entire Lipschitz graphs}

\author{Albert CHAU}
\address{Department of Mathematics\\
University of British Columbia\\
Vancouver, B.C., V6T 1Z2\\
Canada}
\email{chau@math.ubc.ca}

\author{Jingyi CHEN}
\email{jychen@math.ubc.ca}

\author{Weiyong He}
\email{whe@math.ubc.ca}
\thanks{The first two authors are partially supported by NSERC, and the third author is
partially supported by a PIMS postdoctoral fellowship.}
\date{\today}

\begin{abstract}
We consider the mean curvature flow of entire Lagrangian graphs with Lipschitz continuous initial data.  Assuming only a certain bound on the Lipschitz norm of an initial entire Lagrangian graph in $\R^{2n}$, we show that the parabolic equation \eqref{PMA} has a longtime solution which is smooth for all positive time and satisfies uniform estimates away from time $t=0$.  In particular, under the mean curvature flow \eqref{LMCF}  the graph immediately becomes smooth and the solution exists for all time such that the second fundamental form decays uniformly to $0$ on the graph as $t\to \infty$.   Our assumption on the Lipschitz norm is equivalent to the assumption that the underlying Larangian potential $u$ is uniformly convex with its Hessian bounded in $L^{\infty}$.  We apply this result to prove a Bernstein type theorem for translating solitons, namely that if such an entire Lagrangian graph is a smooth translating soliton, then it must be a flat plane.  We also prove convergence of the evolving graphs under additional conditions.  See Theorem \ref{entire} and Theorem \ref{compact} for details.   
\end{abstract}

\maketitle

\section{Introduction}

In this paper we consider the mean curvature flow of entire Lagrangian graphs with Lipschitz continuous initial data. In particular, we focus on existence of smooth solutions assuming only an $L^{\infty}$ bound on the Hessian of the underlying potential of an entire Lagrangian graph.  In the compact case, a shorttime smooth solution of mean curvature flow always exists if the initial submanifold is smooth or satisfies a certain smallness condition on local slope \cite{W1}. In the noncompact case, 
a remarkable result asserting existence of longtime smooth solutions of the mean curvature flow for codimension one entire graphs for any locally Lipschitz initial data, without any growth condition at infinity, is obtained in \cite{HE}. However, for higher codimensions one does not expect a shorttime smooth solution if the initial data has only bounded Lipschitz norm in view of the example of Lawson and Osserman \cite{LO}.

 Consider the fully nonlinear parabolic equation on $\R^n$:
 
\begin{equation}\label{PMA}
\left\{%
\begin{array}{ll}
 &\dfrac{du}{dt} = \dfrac{1}{\sqrt{-1}} \log \dfrac{\det (I_n+ \sqrt{-1}D^2u)}{\sqrt{\det (I_n+ (D^2u)^2)}}\\
 &u(x, 0)=u_0(x),
\end{array}%
\right.
 \end{equation} 
 where $I_n$ is the $n$ dimensional identity matrix. There exists a family of diffeomorphisms $r_t:\R^n \to \R^n$ such that $F(x, t)=(r_t(x), Du(r_t(x), t))\subset \R^{2n}$ is a solution to the mean curvature flow equation
 \begin{equation}\label{LMCF}
\left\{%
\begin{array}{ll}
 & \dfrac{dF}{dt} =  H\\
 &F(x, 0)=F_0(x)
\end{array}%
\right.
\end{equation}
where $H(x, t)$ is the mean curvature vector of the submanifold $F(x, t) \subset \mathbb{R}^{2n}$
at $F(x, t)$ (cf.  Lemma 2.1).

One of our main results in this paper is


\begin{thm}\label{entire}
Suppose that $u_0:\R^n \to \R$  is a function with $L^{\infty}$ Hessian satisfying 
\begin{equation}\label{hesscond}
-(1-\delta) I_n\leq \essinf D^2u_0\leq \esssup D^2u_0 \leq (1-\delta) I_n
\end{equation}
for any $\delta \in (0, 1)$.  Then
\eqref{PMA} has a longtime smooth solution $u(x, t)$ for all $t>0$ with initial condition $u_0$ such that the following estimates hold:

\begin{enumerate}
\item$-(1-\delta) I_n\leq D^2u\leq (1-\delta) I_n$  for all $t>0$,
\item$\sup_{x\in \R^n}|D^l u(x,t)|^2 \leq C_{l, \delta}/t^{l-2} $  for all $l\geq3$, and some $C_{l, \delta}$ depending only on $l$ and $\delta$.
\item $u(x, t)\in C^{\infty}(\R^n \times (0, \infty)) \cap C_{loc}^{1+\alpha, \beta}(\R^n \times [0, \infty))$ for any $0<\alpha, \beta <1$.
\end{enumerate}
If we assume in addition that $|D u_0(x)| \to 0$ as $|x|\to \infty$, then $\sup_{x\in \R^n}|D u(x,t)| \to 0$ as $t\to \infty$.  In particular, the graph $(x, Du(x, t))$ immediately becomes smooth and converges smoothly on compact sets to the coordinate plane $(x, 0)$ in $\R^{2n}$. 
\end{thm}
\begin{rem}\label{ir1}
After a change of coordinates described in \cite{Y}, (1.3) can be restated as: the Hessian of $u$ has positive lower and upper bounds almost everywhere.  
\end{rem}

 As an application of Theorem \ref{entire} we prove the following Bernstein type theorem for entire graphical translating solitons to the mean curvature flow (see Section 2 for definitions).   The theorem can be compared to Yuan's  Bernstein theorem in \cite{Y}  for entire Lagrangian graphs which states that   a minimal graph $(x, Du_0(x)) \subset \R^{2n}$ satisfying $-I_n\leq D^2u\leq I_n$ must be a flat plane.  
   
  \begin{thm}\label{bernstein}
 Suppose that $u_0:\R^n \to \R$ is smooth and satisfies \eqref{hesscond} for any $\delta \in (0, 1)$ and that the graph $(x, Du_0(x))$ is a Lagrangian translating soliton.  Then $(x, Du_0(x))$ must be a flat plane in $\R^{2n}$.
 \end{thm}

Longtime existence and convergence results have been proven in \cite{Sm} and \cite{SmMt} for compact graphical Lagrangian submanifolds of $T^{2n}$, the standard $2n$ dimensional flat torus, with smooth convex initial potentials.   On the other hand, in \cite{HE2} it was shown that \eqref{LMCF} has a longtime smooth solution with curvature decay emerging from any entire hypersurface  in $\R^{n+1}$ with bounded Lipschitz constant and our results can be viewed as a Lagrangian version of this.  When the initial data has sufficiently small Lipschitz bound then the mean curvature flow admits a longtime graphical solution \cite{L}.

One of our key observations is that in the setting of Theorem \ref{entire}, the mean curvature $H$ controls the second fundamental form $A$ along the Lagrangian mean curvature flow where the Bernstein type theorem in  \cite{Y} for special Lagrangian entire graphs plays a crucial role.  In particular, we show that $|H|^2$ decays in time at a rate $C/t$ and thus so does $|A|^2$. This estimate ultimately provides the longtime existence in Theorem \ref{entire}, and allows us to handle initial data without any curvature assumptions.  More precisely, we make no assumption on the $C^{\alpha}$ norm of $D^2 u_0$. 
The strategy of our proof for Theorem \ref{entire} is as follows. In Section 3, we derive the a priori estimate $|H|^2\leq C/t$ for smooth solutions of the mean curvature flow (\ref{LMCF}) under the assumption that the geometry of the graph at each $t\geq 0$ is bounded. This is done by adapting Hamilton's maximum principle for tensors \cite{H} to the noncompact situation. In Section 4 we deal with (\ref{PMA}) for smooth initial potential functions under the assumption that the spacial derivatives $D^l u$ are uniformly bounded in space for every time for $l\geq 2$.  We also show that the initial condition \eqref{hesscond} is preserved under the evolution.  Then applying Yuan's Bernstein theorem \cite{Y} to a blow-up limit, which is minimal if $|H|<C$ along the flow, we obtain the curvature estimate: $D^3 u$ is uniformly bounded along the flow.  From this we then work out the higher derivative estimates on $u$ directly. In Section 5, we first 
 construct a sequence of smooth functions $u_0^k$ approximating $u_0$ such that $D^2u_0^k$ satisfies the desired bounds for all $k$.  We then prove a general shorttime existence result for the graphical mean curvature flow equation \eqref{GMCF} in Proposition \ref{ap1}, and we use this and our a priori estimates to show that for each approximting function $u_0^k$, there exists a smooth longtime solution satisfying the bounded geometry conditions.  Finally the estimates obtained will in fact allow us to extract a convergent subsequence of solutions which converges to a longtime solution to \eqref{PMA}  which is smooth for $t>0$ and converges to $u_0$ as $t\to 0$.  Note that an entire uniformly bounded longtime smooth solution to a parabolic equation may fail to converge as $t\to\infty$ even for the standard heat equation.  Using a theorem of Il'in \cite{I}, we have that if $|Du_0|\to 0$ as $|x|\to\infty$ in addition to \eqref{hesscond}, then the evolving graph $(x,Du(x, t))$ in Theorem \ref{entire} converges to the coordinate plane $(x,0)$ as $t\to \infty$.  We prove Theorem \ref{entire} and Theorem \ref{compact} in Section 6 and we prove Theorem \ref{bernstein} in Section 7.

A particular case of Theorem \ref{entire} is when $Du_0:\R^n \to \R^n$ is a lift of a map $f:T^n \to T^n$,  where $T^n$ is the standard $n$-dimensional flat torus.  In this case we get convergence as in the following 
\begin{thm}\label{compact}
Let $u_0:\R^n \to \R$ satisfy 
\eqref{hesscond} for any $\delta\in(0,1)$ such that  $Du_0:\R^n \to \R^n$ is a lift of a map $f:T^n \to T^n$.  Then (1)-(3) in Theorem \ref{entire} hold.   Moreover, the graph $(x, Du(x, t))$  immediately becomes smooth after initial time and converges smoothly to a flat plane in $\R^{2n}$. 
\end{thm}

By Remark \ref{ir1}, Theorem \ref{compact} generalizes the main result in \cite{SmMt} by allowing nonsmooth initial data; in particular, one does not need to assume the initial curvature is bounded:

\begin{cor}\label{compactentire}
Let $u_0:\R^n \to \R$ satisfying $0< \essinf D^2u_0<\infty$ such that  $Du_0:\R^n \to \R^n$ is the lift of a map $f:T^n \to T^n$.  Then the conclusion of Theorem \ref{compact} holds.  
\end{cor}

\section{Preliminaries}
Let $(x_1 ,\cdot \cdot \cdot, x_n, y_1, \cdot \cdot \cdot, y_n)$ be global coordinates on $\mathbb{R}^{2n}=\mathbb{R}^n\times \mathbb{R}^n$ and define the complex structure tensor $J$ on $\R^{2n}$ in these coordinates by

 \[
J\frac{\partial }{\partial x_i}=\frac{\partial}{\partial y_i}, \hspace{12pt} J\frac{\partial }{\partial y_i}=-\frac{\partial}{\partial x_i}.
\]
We denote the standard Euclidean metric in the above coordinates on $\R^{2n}$ by $\langle\cdot, \cdot\rangle$ and the standard symplectic form $\omega$ by 

\[
\omega=\sum_{i=1}^n dx_i \wedge dy_i.
\]
 
Now let $F(x, t):\R^n \to \R^{2n}$ be a family of smooth immersions for $t \in  [0, T)$, and fix global coordinates $x_1 ,\cdot \cdot \cdot, x_n, y_1, \cdot \cdot \cdot, y_n$ on the target $\R^{2n}$ as above.  We will also use $x_1,..., x_n$ to denote the coordinates on the domain $\R^n$.   Adopting the above definitions and notations for the target $\R^{2n}$, we have the following time dependent tensors on $\R^n$ induced by $F(x, t)$: 

\begin{enumerate}
\item $g_{ij} := \langle F_i,  F_j\rangle$ (the metric tensor)
\item $h_{ijk}:= -\omega( F_i, \nabla_j F_k)$ (the second fundamental form)
\item $H_i:={g}^{jk}h_{ijk}$ (the mean curvature form)
\end{enumerate}
where $F_i:= \partial F/\partial x_i$ and $\nabla$ is the covariant derivative on $\R^n$ with respect to the induced metric $g_{ij}$.  For any vector $V=(v_1 ,\cdot \cdot \cdot, v_n, \tilde  v_{1}, \cdot \cdot \cdot, \tilde  v_{n})$ on $\R^{2n}$ we define 
 $\overline{V}:= (v_1 ,\cdot \cdot \cdot, v_n, - \tilde v_{1}, \cdot \cdot \cdot, -\tilde  v_{n})$.  
 Define the tensor $S$ by
 \[
 S(V, W)=\langle \overline{ V}, W \rangle
 \]
 and consider the corresponding tensor on $\R^n$ induced by $F(x, t)$:
\[S_{ij}:=\langle\overline {F}_i, F_j\rangle.\]
The above tensor was introduced in \cite{Sm} and as in the case there, it will play a key role in our a priori estimates.  
 
An immersion $F:\R^n \to \R^{2n}$ is called Lagrangian provided $F^* \omega =0$  on $F(\R^n)$. It is not hard to show that given any function $u:\R^n \to \R$, the corresponding graph $(x, Du(x))\subset \R^{2n}$ is always Lagrangian and thus any family of functions $u(x, t):\R^n \to \R$ defines a family of Lagrangian graphs in $\R^{2n}$.   The following lemma from \cite{Sm} establishes the correspondence between solutions of \eqref{PMA} and \eqref{LMCF} sufficient for our purposes.
 
 \begin{lem}\label{pl1}
Let $u_0:\R^n \to \R$ be a smooth function.  Then \eqref{PMA} has a smooth solution on $\R^n \times [0, T)$ with initial condition $u(x, 0)=u_0$ if and only if  \eqref{LMCF} has a smooth solution $F(x, t)$ on $\R^n \times [0, T)$ with initial condition $F(x, 0):=(x, Du_0(x))$.  
 In particular, there exists a smooth family of diffeomorphisms $r(x, t):\R^n \to \R^n$ for $t\in [0, T)$ such that $F(x, t):=(r(x, t), Du(r(x, t), t))$ solves \eqref{LMCF} on  $\R^n \times [0, T)$.
\end{lem}

 A Lagrangian submanifold $L$ in $\mathbb{R}^{2n}=\mathbb{C}^n$ is called  {\it special Lagrangian} when its mean curvature is identically zero in which case it is a stationary solution of \eqref{LMCF}.  As in \cite{Harvey-Lawson}, an entire special Lagrangian graph $(x, Du(x))$ satisfies the equation
\begin{equation}\label{E-special-lagrangian}
F(D^2u)=\sum_i\arctan \lambda_i=\Theta
\end{equation}
for some constant $\Theta$, where $\lambda_i$s are eigenvalues of $D^2u$.   More generally, $L$ is called a {\it Lagrangian translating soliton} when $L_t=L+tT$ is a solution to the mean curvature flow for some constant {\it translating vector} \[T=(a_1, a_2, \cdots, a_n, b_1, b_2, \cdots, b_n).\]   An entire Lagrangian translating soliton graph $(x, Du(x))$ with translating vector $T$ satisfies the equation
\begin{equation}\label{E-soliton-1}
\sum_i\arctan \lambda_i+\sum_i a_i\frac{\partial u}{\partial x_i}-\sum_ib_ix_i=c,
\end{equation}
 for  some constant $c$.  We can derive \eqref{E-soliton-1} as follows.  In general, a translating soliton to the mean curvature flow is defined by the identity
$H\equiv T^{\bot}$, where $H$ is the mean curvature vector of the graph $(x, Du(x))$ and $T^{\bot}$ is the normal part of a constant vector $T$ along the submanifold.   On the other hand, on a Lagrangian submanifold $L$, the mean curvature vector $H$ is given by
\[
H=J\nabla \theta,
\]
where $\theta=\sum \arctan \lambda_i$ is the {\it Lagrangian angle}.  Thus when $L$ is a Lagrangian translating soliton we have
\begin{equation}\label{E-soliton-1.1}
(\nabla \theta, F_i) =(-J T^{\bot}, F_i).
\end{equation}

Now \eqref{E-soliton-1} then follows from  \eqref{E-soliton-1.1} and the fact that $(\nabla \theta, F_i)=\frac{\partial \theta}{\partial x_i}$ and $(-JT^\bot, F_i)=(-J T, F_i)=b_i-\sum_ka_ku_{ik}$.

\section{A priori estimates for (\ref{LMCF})}
In this section we establish some a priori estimates for smooth solutions to  (\ref{LMCF}). 
Assume that $F(x, t)$ is a smooth solution to \eqref{LMCF}  on $\R^{n}\times[0, T)$ for some $0<T\leq \infty$
such that  $F(\R^n, t)\subset \R^{2n}$ satisfies:
\begin{equation}\label{boundgeom}
\sup_{x\in{\mathbb R}^n} |\nabla^k A(x,t)|\leq C(t,k) <\infty
\end{equation}
for each $t<T$ and nonnegative integer $k$, where $C(t,k)$ is a positive constant which may go to infinity as $t$ tends to $T$ for a fixed $k$. We also assume that the pullback $F(\cdot, t)^{*}ds^2$ of the Euclidean metric $ds^2$ on $\R^{2n}$ is equivalent to the Euclidean metric $dx^2$ on $\R^n$ for any $t\in [0, T)$:
\begin{equation}\label{metric}
C_1(t)dx^2 \leq F(\cdot,t)^*ds^2 \leq C_2(t) dx^2
\end{equation}
for some positive constants $C_1(t),C_2(t)$ depending only on $t$.  Such $F(x, t)$, satisfying \eqref{boundgeom} and \eqref{metric}, is said to have {\it bounded geometry} for every $t\in [0, T)$.

Let $\ttD f := {g}^{ij} \ttN_{i}\ttN_{j}f$ for any function $f$ where $\ttN$ is the covariant derivative relative to $\tg$.  
We establish the following lemma which is a noncompact version of Lemma 3.1 in \cite{Sm}.

\begin{lem}\label{L-3-1}Let $F(x, t)$ be a smooth solution to \eqref{LMCF} and suppose that $F(x, t)$ has bounded geometry for each $t\in [0, T)$. For any given $\epsilon>0$, if $S_{ij}-\epsilon  g_{ij}\geq 0$ at $t=0$ then $S_{ij}-\epsilon g_{ij}\geq 0$ for all $t\in[0,T)$.
\end{lem}

\begin{proof}
Recall the following formulas from \cite{Sm}:
\begin{equation}
\begin{split}
\frac{d}{dt} S_{ij} &= \ttD S_{ij} -R^{l}_i S_{lj} -R^l_j S_{li} +2h^{km}_i h^{n}_{jk} S_{mn}\\
\frac{d}{dt}\tg &= -2H^l h_{lij}
\end{split}
\end{equation}
where $h^i_{jk}:=g^{in}h_{njk}$, $h^{ij}_{k}:=g^{in}g^{mj}h_{mnk}$, $R_{i}^l:= g^{lj} R_{jl}$ and $R_{ij}=H^lh_{lij} -h^{mn}_i h_{mnj}$ is the Ricci curvature tensor of $g_{ij}$.  It follows that
\begin{equation*}
\left(\frac{\partial }{\partial t}-\ttD\right) \left(S_{ij}-\epsilon  g_{ij}\right)=-R^{l}_i \left(S_{lj}-\epsilon  g_{lj}\right) -R^l_j \left(S_{li}-\epsilon  g_{li}\right) +2h^{km}_i h^{n}_{jk} S_{mn}+2\epsilon h_{i}^{mn}h_{jmn}.
\end{equation*}

We will adopt Hamilton's maximum principle for tensors to the noncompact case.

\vspace{.2cm}

\noindent{\bf Step 1.}  We first show that $S_{ij}-\epsilon g_{ij}\geq 0$ in $[0,T)$ if $S_{ij}-\epsilon  g_{ij}/2\geq 0$ in $[0, T)$.  

\vspace{.2cm}

For constants $\delta,  R >0$ and $0<\mu<\epsilon$ to be determined, consider the function 
$$
\phi_R(x):= 1+ \left|\frac{x}{R}\right|^2
$$
and consider the symmetric tensor \[a_{ij}:=e^{\delta t} \phi_R S_{ij} - (\epsilon-\mu )  g_{ij}.\]
We compute
\begin{equation}\label{E-3-2}
\begin{split}
\frac{\partial}{\partial t} a_{ij}=&e^{\delta t}\phi_R \left(\ttD S_{ij}-R^{l}_i S_{lj} -R^l_j S_{li} +2h^{km}_i h^{n}_{jk} S_{mn}\right)\\
 & +2(\epsilon-\mu)H^lh_{lij}+\delta e^{\delta t}\phi_R S_{ij}\\
 =&\ttD a_{ij}- R^{l}_i a_{lj}-R^l_ja_{li}+2 (\epsilon- \mu ) h_i^{mn}h_{jmn}+2e^{\delta t}\phi_R h^{km}_i h^{n}_{jk} S_{mn}\\
 &+\delta e^{\delta t}\phi_R S_{ij}-e^{\delta t} \ttD (\phi_R) S_{ij} -2 e^{\delta t} \nabla\phi_R\nabla S_{ij}.
\end{split}
\end{equation}

Fix $0<T'<T$.  Now we want to prove that for any $\delta, \mu>0$ there exists a large number $R_0$ such that $a_{ij}>0$  for all $R\geq R_0$ in $[0, T^{'}]$. 
It is clear that $a_{ij}>0$ on $\R^n\times\{0\}$.  Moreover, since $S_{ij}-\epsilon  g_{ij}/2\geq 0$ in $[0, T)$,  there exists $r>0$ such that for $|x|\geq r$ we have $a_{ij}(x, t)>0$ for any $t\in [0, T^{'}]$.  Suppose now that $a_{ij}$ has a zero eigenvalue at $(x_0, t_0)\in \R^n\times [0, T']$.   Then we have  $(x_0, t_0)\in \overline{B_{r}(0)}\times (0, T^{'}]$, and we may further assume that  $t_0$ is the first such time:  $a_{ij}>0$ for any $t<t_0$.  Let $V=(V^1,...,V^n)\not=0$ be a null vector of $a_{ij}(x_0, t_0)$, namely $a_{ij}(x_0, t_0)V^j=0$. We extend $V$ locally as follows: at the time slice $t_0$ we parallel translate $V$ along radial geodesics in a normal neighborhood $U$ around $x_0$, then set $V(x,t)=V(x,t_0)$ for $x\in U$ and $t\leq t_0$. Then at the point $(x_0, t_0)$, we have
\begin{equation}\label{E-3-2.9}
\nabla V = \nabla^2 V =0
\end{equation}
and
\begin{equation}\label{E-3-3}
\frac{d}{dt} (a_{ij} V^iV^j)\leq0 \hspace{12pt} \mbox{and} \hspace{12pt} \ttD (a_{ij} V^iV^j)\geq0.
\end{equation}
Straightforward computation leads to 
\begin{equation}
\begin{split}
\nabla_i \phi_R&= g^{ik}\frac{2 x_k}{R^2}\\
\ttD \phi_R&= \frac{2}{R^2}\sum_kg^{kk}+ \frac{2}{R^2\sqrt{g}}\partial_i(g^{ik}\sqrt{g}) x_k\\
\end{split}
\end{equation}
By assumption (\ref{metric}) $g_{ij}$ is uniformly equivalent to the Euclidean metric on $\R^n$ in $[0, T']$ up to a constant depending only on $T'$.  As noted in \cite{Sm},  $S_{ij}$ is the pullback of a constant coefficient tensor $S$ on $\R^{2n}$ by $F(x, t)$.  It follows by the bounded geometry assumption on $F(x, t)$ that $S_{ij}$ and $\nabla S_{ij}$ are uniformly bounded on $\R^n\times [0, T']$ by a constant depending only on $T'$. Thus at $(x_0,t_0)$
we have 
\begin{equation}\label{laplace}
\begin{split}
-c(T')\frac{|x_0|}{R^2} &S_{ij}\leq \nabla\phi_R\nabla S_{ij} \leq  c(T')\frac{|x_0|}{R^2} S_{ij}\\
&\left|\ttD \phi_R\right| \leq c(T')\left(\frac{1}{R^2}+\frac{|x_0|}{R^2}\right)\\
\end{split}
\end{equation}
where $c(T')$ is a constant depending only on $T'$ and we have used the assumption that $S_{ij}\geq \epsilon g_{ij}/2$. 
Now choose sufficiently large $R_0$  depending on $T'$ and $\delta$ such that the following real quadratic in $y$ is positive for any $y$: $$\frac{\delta}{2} y^2 -\frac{3c(T')}{R_0} y +\left(\frac{\delta}{2}-\frac{c(T')}{R_0^2}\right)>0.$$  For any $R\geq R_0$, the following quadratic in $y$ is also positive
 \[
 \frac{\delta}{2} y^2 -\frac{3c(T')}{R} y +\left(\frac{\delta}{2}-\frac{c(T')}{R^2}\right)>0.
 \]
  Then we compute
\begin{equation}\label{E-3-new}
\begin{split}
&\delta e^{\delta t}\phi_R S_{ij}-e^{\delta t} \ttD (\phi_R) S_{ij} -2 e^{\delta t} \nabla\phi_R\nabla S_{ij}\\
&\geq e^{\delta t}S_{ij}\left(\delta +\delta \frac{|x_0|^2}{R^2}-3\frac{c(T^{'})|x_0|}{R^2}-\frac{c(T^{'})}{R^2}\right)\\
&\geq e^{\delta t}S_{ij}\left[\frac{\delta}{2}\left(\frac{|x_0|}{R}\right)^2-3\frac{c(T^{'})}{R}\left(\frac{|x_0|}{R}\right)+\frac{\delta}{2}-\frac{c(T^{'})}{R^2}\right]+\frac{\delta}{2}e^{\delta t}S_{ij}\\
&>\frac{\delta}{2}e^{\delta t}S_{ij}.
\end{split}
\end{equation}
Now denote \[N_{ij}=
\frac{\partial}{\partial t} a_{ij}-\ttD a_{ij}.
\]
Then for any  $R\geq R_0$ and $\mu<\epsilon$,  
by \eqref{E-3-2}, \eqref{E-3-new} and the fact that $V$ is parallel around $x_0$ when $t=t_0$, the following holds at $(x_0, t_0)$
\begin{equation}
\begin{split}
N_{ij}V^iV^j=&2(\epsilon-\mu)h_i^{mn}h_{jmn}V^iV^j+2e^{\delta t}\phi_R h^{km}_i h^{n}_{jk} S_{mn}V^iV^j\\
&+\left(\delta e^{\delta t}\phi_R S_{ij}-e^{\delta t} \ttD (\phi_R) S_{ij}  -2 e^{\delta t} \nabla\phi_R\nabla S_{ij}V^iV^j\right)\\
> & \frac{1}{2}\delta e^{\delta t}S_{ij}V^iV^j\\
>&0.\\
\end{split}
\end{equation}
But this contradicts \eqref{E-3-3}. It follows that $a_{ij}>0$ in $[0, T^{'}]$. Now let  $R\rightarrow \infty$ first, 
then $\mu\rightarrow 0$, and finally $\delta\rightarrow 0$, we get that
\[
S_{ij}-\epsilon  g_{ij}\geq 0.
\]
Since this holds for any $T^{'}$, we have proved that $S_{ij}-\epsilon {g}_{ij}\geq 0$ holds in $[0, T)$ under the assumption $S_{ij}-\epsilon g_{ij}/2\geq 0$ in $[0,T)$.

\vspace{.2cm}

\noindent{\bf Step 2.} We now remove the assumption $S_{ij}-\epsilon { g}_{ij}/2\geq 0$ in $[0, T)$ in Step 1. 

\vspace{.2cm}

First note that  at $t=0$
 $$
 S_{ij}-\frac{\epsilon}{2}g_{ij}\geq\frac{\epsilon}{2} {g}_{ij}\geq C\epsilon\, \delta_{ij}
 $$ 
 for some constant $C>0$. Also by the bounded geometry assumption on $F(x, t)$ we know that for any $t\in [0, T)$ 
$$
\left|\frac{d}{dt} S_{ij}(x,t)\right|\leq C(t)
$$ 
where $C(t)$ is a constant depending only $t$. It follows that there is a maximal positive time $T_0$, such that 
$S_{ij}-\epsilon{g}_{ij}/2>0$ holds in $[0, T_0)$.  
Then by the maximum principle argument above, we know that $S_{ij}-\epsilon {g}_{ij} \geq 0$ 
 in $[0, T_0)$.  If $T_0\neq T$, by continuity, we know that $S_{ij}-\epsilon {g}_{ij}\geq 0$ in $[0, T_0]$ and we can then find some positive $T_0'$ such that $S_{ij}-\epsilon {g}_{ij}/2$ holds in $[T_0, T_0+T_0')\subset[T_0,T)$.  But this contradicts the choice of $T_0$. So $T_0=T$.  \end{proof}

In the compact case, it is proved in \cite{Sm} that $|H|\leq C$ is preserved along the mean curvature flow \eqref{LMCF} by considering the tensor $S_{ij}-\epsilon H_iH_j$.  Inspired by \cite{Sm}, we use $S_{ij}-\epsilon t H_iH_j$ to obtain the global decay estimate on the mean curvature in time: $|H|^2 \leq Ct^{-1}$.  These estimates hold in both the compact and noncompact case.

\begin{lem}\label{apel1}Let $F(x, t)$ be a smooth solution of \eqref{LMCF}  having bounded geometry for each $t\in [0, T)$. 
Suppose $S_{ij}-\epsilon_1 \tg \geq 0$ on $\R^{n}\times[0, T)$ for some $\epsilon_1>0$.  Then there exists a constant $\epsilon_2 >0$ depending only $\epsilon_1$ such that
\[S_{ij} -\epsilon_2 t H_i H_j \geq 0\] on $\R^{n}\times[0, T)$.
\end{lem}
\begin{proof}   
Recall the following formulas from \cite{Sm}:
\begin{equation}\label{ape4}
\begin{split}
\frac{d}{dt} S_{ij} &= \ttD S_{ij} -R^{l}_i S_{lj} -R^l_j S_{li} +2h^{km}_i h^{n}_{jk} S_{mn}\\
\frac{d}{dt}H_i &= \ttD H_i -R^j_i H_j\\
\frac{d}{dt}\tg &= -2H^l h_{lij}
\end{split}
\end{equation}

Now for any $\delta, R >0$ and $\epsilon_2 >0$ to be determined, let $\phi_R(x):= 1+ |x/R|^2$ and consider the tensor \[M_{ij}:=e^{\delta t} \phi_R S_{ij} - \epsilon_2 t H_i H_j.\]
Then using \eqref{ape4} we calculate: 
\begin{equation}\label{ape5}
\begin{split}
\frac{d}{dt} M_{ij} =&e^{\delta t} \phi_R ( \ttD S_{ij} -R^{l}_i S_{lj} -R^l_j S_{li} +2h^{km}_i h^{n}_{jk} S_{mn})+\delta e^{\delta t} \phi_R S_{ij} \\
& - \epsilon_2 t (\ttD H_i -R^l_i H_l) H_j -\epsilon_2 t H_i (\ttD H_j -R^l_j H_l) -  \epsilon_2 H_i H_j\\
= &\ttD M_{ij} +2\epsilon_2 t \ttN H_i \ttN H_j  -R^{l}_i M_{lj} -R^l_j M_{li}\\
& -e^{\delta t} \ttD (\phi_R) S_{ij} -2e^{\delta t} \nabla \phi_R \nabla S_{ij} +\delta e^{\delta t} \phi_R S_{ij}\\& +2e^{\delta t} \phi_R h^{km}_i h^{n}_{jk} S_{mn} -  \epsilon_2 H_i H_j \\
\end{split}
\end{equation}

  Fix $0<T'<T$.  Now we want to prove that for any $\delta, \mu>0$ there exists a large number $R_0$ such that $M_{ij}>0$  for all $R\geq R_0$ in $[0, T^{'}]$.  It is clear that $M_{ij}(x, 0)>0$, and there exists $r$ such that for $|x|\geq r$ we have $M_{ij}(x, t)>0$ for any $t\in [0, T^{'}]$.  Suppose now that $M_{ij}$ has a zero eigenvalue at $(x_0, t_0)\in \R^n\times [0, T']$.   Then we have  $(x_0, t_0)\in \overline{B_{r}(0)}\times (0, T^{'}]$, and we may further assume that  $t_0$ is the first such time:  $M_{ij}>0$ for any $t<t_0$.  Let $V$ be a null vector for $M_{ij}(x_0, t_0)$ and extend $V$ locally in space and time as in the proof of Lemma \ref{L-3-1}. Then at $(x_0, t_0)$ we have (3.5) and 
\begin{equation}\label{E-3-4}
\frac{d}{dt} (M_{ij} V^iV^j)\leq0 \hspace{12pt} \mbox{and} \hspace{12pt} \ttD (M_{ij} V^iV^j)\geq0.
\end{equation}
We estimate the following at  $(x_0, t_0)$
\begin{equation}\label{ape7}
\begin{split}
N_{ij}V^i V^j:=&(2\epsilon_2 t \ttN H_i \ttN H_j  -R^{l}_i M_{lj} -R^l_j M_{li}  
-e^{\delta t} \ttD (\phi_R) S_{ij} \\& +2e^{\delta t} \phi_R h^{km}_i h^{n}_{jk} S_{mn} +\delta e^{\delta t} \phi_R S_{ij} -  \epsilon_2 H_i H_j )V^i V^j\\
\geq&( -e^{\delta t} \ttD (\phi_R) S_{ij} -2 e^{\delta t} \nabla\phi_R\nabla S_{ij}+\delta e^{\delta t} \phi_R S_{ij})V^iV^j\\
& +(2e^{\delta t} \phi_R h^{km}_i h^{n}_{jk} S_{mn}  -  \epsilon_2 H_i H_j )V^i V^j \\
\end{split}
\end{equation}

Now fix $\delta>0$.  Then as in the proof of Lemma 3.1 (see \eqref{E-3-new}), there exists $R_0$ depending on $T^{'}$ and $\delta$ such that for $R\geq R_0$, we have
\[
 -e^{\delta t} \ttD (\phi_R) S_{ij}-2 e^{\delta t} \nabla\phi_R\nabla S_{ij} +\delta e^{\delta t} \phi_R S_{ij} >\frac{1}{2}\delta e^{\delta t}  S_{ij}>0.
\]
Also, for some choice of $\epsilon_2$ depending only on $\epsilon_1$, the term in the last line in \eqref{ape7} is nonnegative since $S_{ij}\geq \epsilon_1 g_{ij}$.   Thus at $(x_0, t_0)$ we have, for these choices of constants 
 \begin{equation}\label{ape8}
\frac{d}{dt} (M_{ij} V^iV^j)=\ttD  (M_{ij} V^iV^j) + N_{ij} V^i V^j>0,
\end{equation}
which contradicts \eqref{E-3-4}.   We have thus proved that for the above choices of constants, $M_{ij}>0$ on $\R^n\times [0 ,T']$.   The lemma then follows from first letting $R \to \infty$ then letting $\delta \to 0$.
 \end{proof}

\begin{cor}\label{apec1}
Under the hypothesis in Lemma \ref{apel1} we have $|H|^2 \leq C/t$ on $\R^n \times [0, T)$ for some constant $C$ depending only on $\epsilon_1$.
 \end{cor}

\begin{proof}
By the definition of $S$ on $\R^{2n}$, it is not hard to see that ${g}^{ij}S_{ij}\leq n$ everywhere on $\R^n$.  The desired result now follows immediately by tracing the tensor in Lemma \ref{apel1} with respect to the metric $g_{ij}$.
\end{proof}

This decay estimate is crucial in this paper. With the blow-up argument in next Section, it allows us to deal with initial data $u_0$ with Hessian in $L^{\infty}$ and obtain a longtime solution to \eqref{PMA} with curvature decay at $t=\infty$. 
  
\section{A priori estimates for (\ref{PMA})}
In this section we prove some a priori estimates for smooth solutions to \eqref{PMA}.  Let $u(x, t)$  be a smooth solution to \eqref{PMA} on $\R^{n}\times[0, T)$ for some $0<T\leq \infty$. We assume  that for every $k\geq 2$ we have 
$$
\sup_{\R^n} |D^k u(x, t)|\leq C(k, t)<\infty
$$ 
for any $t\in [0, T)$ where $C(t,k)$ is a positive constant which may go to infinity as $t$ tends to $T$ for a fixed $k$.  We will say that such a solution $u(x, t)$ satisfies the \begin{it}bounded geometry condition\end{it}  for any $t\in [0, T)$. 
Let $\tilde g_{ij}$ be the metric on $\R^n$ induced by 
$$
\tilde F(x, t):=(x, Du(x, t))
$$ 
and introduce the operator $L$ defined by  
$$
L f=\tilde g^{ij}f_{ij}
$$ 
for any smooth function $f$ on $\R^n$ with subscripts denoting partial differentiation with respect to the standard coordinates on $\R^n$.  Thus $\tilde g_{ij}=\delta_{ij}+u_{ik}u_{kj}$ and we note that $\tilde g$ and $g$ differ by a tangential diffeomorphism for each $t$.  We will also let 
$$
|D^ku|^2=u_{i_1,..,i_k}u_{i_1,..,i_k}
$$ 
where we adopt the summation convention for repeated indices.
 Differentiating \eqref{PMA} gives
 \begin{equation}\label{pmaape1}
 \frac{\p u_i}{\p t}=\tilde g^{pq}u_{pqi}=Lu_i.
\end{equation} 
\begin{lem}\label{L-4-1}
Suppose $u(x, t)$ is a smooth solution of \eqref{PMA} satisfying the \begin{it}bounded geometry condition\end{it}. Consider the nonnegative definite matrix $b_{ij}:=u_{ik}u_{kj}$. If $b_{ij}\leq (1-\delta) \delta_{ij}$ initially for some $\delta$ with $0<\delta<1$, then it remains so along the equation \eqref{PMA}.
\end{lem}
\begin{proof}This is a consequence of Lemma \ref{L-3-1} along the mean curvature flow \eqref{LMCF}. Let $u(x, t)$ be a smooth solution of \eqref{PMA}.  By Lemma  \ref{pl1}, the family of Lagrangian graphs $\tilde{F}(x, t)=(x, Du(x, t))$ differs by a tangential diffeomorphism  to a smooth solution $F(x, t)$ of the mean curvature flow \eqref{LMCF} with the initial data $\tilde F_0(x)=(x, Du(x, 0) )$. Let $\epsilon=(1-\delta)/(1+\delta)$.  Now pulling back the ambient tensors $S$ and  $\langle\cdot, \cdot \rangle$ by $F(x, t)$, Lemma \ref{L-3-1} shows that $S_{ij}-\epsilon  g_{ij}\geq 0$ is preserved along the mean curvature flow \eqref{LMCF}.   Now let $\tilde{S}=F^*S$.  Then as $\tilde{F}$ and $F$ differ only by tangential diffeomorphisms we have $\tilde{S}_{ij}-\epsilon  \tilde{g}_{ij}\geq 0$ is preserved along the flow \eqref{PMA}.  Writing this  in terms of the potential $u$ gives $\tilde{S}_{ij}=\delta_{ij}-u_{ki}u_{kj}$ and $\tilde g_{ij}=\delta_{ij}+u_{ki}u_{kj}$ respectively, and thus $\delta_{ij}-u_{ki}u_{kj}-\epsilon (\delta_{ij}+u_{ki}u_{kj})\geq 0$ is preserved along the flow (\ref{PMA}).  This implies that $b_{ij}\leq (1-\delta)\delta_{ij}$ is preserved along \eqref{PMA}.
\end{proof}
A direct consequence of Lemma \ref{L-4-1} is that the relation 
$$
-(1-\delta)I_n \leq D^2u\leq(1-\delta) {I_n}
$$ 
is preserved along \eqref{PMA} for any $\delta\in (0, 1)$.  

Next, we will derive higher order estimates for \eqref{PMA} via a blowup argument. To do so, we will employ a parabolic scaling which we now describe.   Define
\begin{eqnarray*}
 y&=&\lambda (x-x_0),\\
 s&=&\lambda^{2}(t-t_0),\\
 u_{\lambda}(y, s)&=&\lambda^2 \left(u(x, t )- u (x_0, t_0)-D_{x}u(x_0,t_0\right)\cdot (x-x_0)).
\end{eqnarray*}
We compute \[ D^2_{ y}  u_{\lambda}=D_x^2u \,\,\,\,\,\mbox{and}\,\,\,\,\, \frac{\partial}{\partial s} u_{\lambda}=\frac{\partial}{\partial t} u. 
\]
So $ u_{\lambda}(y,s)$ is  a solution of \eqref{PMA} with $ u_{\lambda}(0, 0)=0$ and $D  u_{\lambda}(0,0)=0$. Also we can verify
\[
D^l_{y}  u_{\lambda} ( y, s)=\lambda^{2-l} D^l_{x}u(x, t)
\]
for all nonnegative integers $l$. We refer to $(y,D u_{\lambda}(y,s))$ as the {\it parabolic scaling of the graph $(x, Du(x, t))$ by $\lambda$ at $(x_0, t_0)$}.

\begin{lem}\label{pmaapel4}Let $u$ be a smooth solution of \eqref{PMA} in $[0, T)$  satisfying the {\it bounded geometry condition}. Suppose $|D^2u|^2\leq C$ and $|D^3u|^2\leq C$ on $\mathbb{R}^n \times [0, T)$ for some $C$. Then for every $l \geq 4$ there exists a constant $C_{l}$ such that  
\[\sup_{x\in \R^n} |D^lu(x, t)|^2 \leq C_l\] for all $t\in [0, T)$.\end{lem}
\begin{proof} Instead of appealing to results from the mean curvature flow theory on higher order derivatives of $A$  and then converting them to derivatives of $u$, we argue directly for the equation \eqref{PMA}. 
Suppose in addition $|D^4u|^2\leq C$ in $[0, T)$. Then a standard parabolic bootstrapping argument for the quasilinear equation \eqref{pmaape1} gives $|D^lu|^2\leq C_l$ for $l\geq 5$. It will thus suffice to prove the lemma for $l=4$ which we do below.

Suppose that $|D^4u|$ were not bounded over ${\mathbb R}^n\times[0,T)$. By the 
bounded geometry condition assumption on $u$, there would be a sequence $t_k\to T$ such that  
$$
 2\mu_k:=\sup_{x \in \R^n} |D^4u (x, t_k)|^2 \rightarrow \infty
$$ 
and 
$$
\sup_{t\leq t_k,x\in  \R^n}|D^4u(x, t)|^2\leq 2\mu_k.
$$
 Then there exists $x_k$ such that $|D^4u(x_k, t_k)|^2\geq \mu_k\rightarrow \infty$ for $t_k\rightarrow T$.  Let $(y, Du_{\lambda_k}(y, s))$ be the parabolic scaling of $(x, Du(x, t))$ by $\lambda_k=\mu_k^{1/4}$ at $(x_k, t_k)$ for each $k$.

Thus $ u_{\lambda_k}(y, s)$ is a solution of \eqref{PMA} for $s \in [-\lambda^2_k t_k, 0]$ and the first order derivatives of $u_{\lambda_k}$ satisfies the quasilinear parabolic equation (\ref{pmaape1}):
$$
\frac{\partial (u_{\lambda_k})_p}{\partial s} =\tilde g^{ij}(u_{\lambda_k})_{pij}
$$
Note that
\begin{eqnarray*}
&&|D^2_{y} u_{\lambda_k} |=|D^2_xu|\leq C,\\
&&|D^3_{y} u_{\lambda_k} |^2=\lambda_k^{-2}|D^3_xu|^2\rightarrow 0\,\,\,\mbox{as}\,\,\, k\to\infty
\end{eqnarray*}
and 
\begin{eqnarray*}
&&|D^4_{y} u_{\lambda_k} |^2\leq \lambda_k^{-4}|D^4_xu|^2\leq 2, \\
&&|D^4_{y} u_{\lambda_k} (0, 0)|\geq 1.
\end{eqnarray*}
By the parabolic bootstrapping argument, $|D_{y}^l u_{\lambda_k}|$ are uniformly bounded for $s\in [-\lambda_k^{2}t_k, 0], l\geq 5$ and for any $k$. Therefore, the $s$ derivatives of $ u_{\lambda_k}$ of any positive order are uniformly bounded as well. 
Recall that $ u_{\lambda_k}(0, 0)=0$ and $D_{y} u_{\lambda_k}(0,0)=0$, hence in any fixed ball $B_R(0)$ in ${\mathbb R}^n$ there is a positive constant $C$ independent of $k$ and $s$ such that
$$
\left|  u_{\lambda_k} (y,s) \right| \leq C (R^2+|s|^2)\,\,\,\mbox{and}\,\,\,\left|D_{y} u_{\lambda_k}(y,s)\right|\leq C(R+|s|).
$$
Therefore $ u_{\lambda_k}$ converges subsequentially to a smooth function $ u_R$ on $B_R(0)\times [-R,0]$ for any $R>0$, and 
a diagonal sequence argument shows that  $ u_{\lambda_k}$ converges subsequentially and uniformly on compact subsets in $\R^n \times (-\infty, 0]$ to a smooth solution $ u_\infty$ of \eqref{PMA}  with
\[
|D^3_{y}  u_\infty|= 0\,\,\,\mbox{and}\,\,\, |D^4_{y}  u_\infty(0, 0)|\geq 1,\]
which is a contradiction.  Thus $\left|D^4u\right|$ is bounded in $[0,T)$, and hence completes the proof.   
\end{proof}

\begin{lem}\label{pmaapel3}
Let $u$ be a smooth solution of \eqref{PMA} in $[0, T)$  satisfying the bounded geometry condition. Suppose $-(1-\delta)I_n\leq D^2 u(x, 0)\leq (1-\delta)I_n$ for some $\delta\in(0,1)$ and $|H|\leq C$ on $\R^{n}\times[0, T)$ for some constant $C$.  Then for every $l \geq 3$ there exists a constant $C_l$ such that  
\[ \sup_{x\in \R^n} |D^lu(x, t)|^2 \leq C_l \] for all $t\in [0, T)$.
\end{lem}
\begin{proof} By Lemma \ref{L-4-1}, $-(1-\delta)I_n\leq D^2 u \leq (1-\delta)I_n$ holds in $[0,T)$. Also, by Lemma \ref{pmaapel4},
we need only to prove the lemma in the case $l=3$.  Suppose that the lemma were false for $l=3$. Let $$A(t):=\sup_{t'\leq t, x\in \R^n} |D^3u(x, t)|.$$  Then there is a sequence $(x_k, t_k)$ along which we have $|D^3u (x_k, t_k)|\geq A (t_k)/2$ while $A(t_k) \to \infty$ as $t_k\rightarrow T$. 
Let $\lambda_k=A(t_k)$.  For each $k$ let $(y, D_{y} u_{\lambda_k}(y, s))$ be the parabolic scaling of the graph $(x,  Du(x, t))$ by $\lambda_k$ at $(x_k, t_k)$.  Then  $u_{\lambda_k}(y, s)$ is a smooth solution of \eqref{PMA} on $\R^n \times [-\lambda_k^2t_k, 0]$.  Note that
\begin{eqnarray*}
&&\left|D^2_{y}u_{\lambda_k}\right|\leq C\\
&&|D^3_{y}u_{\lambda_k}|=\lambda_k^{-1}|D^3_xu|\leq 1\\
\end{eqnarray*}
on $\R^n \times [-\lambda_k^2t_k, 0]$ and
\[
|D^3_{y} u_{\lambda_k}(0, 0)|\geq \frac{1}{2}.
\]
By Lemma \ref{pmaapel4}, we conclude that all the higher derivatives of $u_{\lambda_k}$ are uniformly bounded on $\R^n \times [-\lambda_k^2t_k, 0]$. As in the proof of Lemma \ref{pmaapel4}, there is a subsequence of $u_{\lambda_k}$ converging smoothly and uniformly on compact subsets in $\mathbb{R}^n\times (-\infty, 0]$ to  a smooth solution  $u_\infty$ to \eqref{PMA} on $\mathbb{R}^n\times (-\infty, 0]$.  Since $|H|\leq C$ for the graphs $(x, Du)$ by assumption, after scaling we have $|H_{\lambda_k}|\leq \lambda_k^{-1} C$ for the graphs $(y, D_{y} u_{\lambda_k} )$.  It follows that the graphs $(y,  D_{y} u_\infty (y, s))$  have $|H_\infty|=0$ everywhere and is then a special Lagrangian graph with $-I_n\leq D_{y}^2  u_\infty\leq I_n$.  Then by the Bernstein Theorem in  \cite{Y}, $u_{\infty}$  is a quadratic polynomial.  This contradicts $|D^3_{y} u_\infty (0, 0)|=1/2$.\end{proof}

We end the section by showing that a bound on the height of the graphs is preserved along \eqref{PMA}. We only need this for the convergence part of Theorem \ref{entire}.
\begin{lem}\label{pmaapel2}
Suppose that $u$ is a smooth solution of (\ref{PMA}) in $[0,T)$ and satisfies the bounded geometry condition. Then 
$$\sup_{x\in \R^n} |Du(x, t)|^2 \leq \sup_{x\in \R^n} |Du(x, 0)|^2$$ for all $t\in [0, T)$.
\end{lem}

\begin{proof}
Multiplying \eqref{pmaape1} by $u_i$ and summing over $i$ gives 
 \begin{equation}\label{pmaape4}
\left(\frac{\p }{\p t} -L\right)|Du|^2=-2ng^{pq}u_{pi}u_{qi}\leq 0.
\end{equation} 
As $u$ satisfies the bounded geometry condition, the lemma follows from the maximum principle in \cite{HE}.
\end{proof}

\section{An approximating sequence}
 In this section we first construct a sequence of smooth approximations, each with bounded geometry,  of the initial data in Theorem \ref{entire}.    Then we establish a general shorttime existence result for the nonparametric mean curvature flow equation \eqref{GMCF}. By  the a priori estimates in Section 3 and Section 4, we obtain a longtime solution to \eqref{PMA} for each approximation.
 
 \begin{lem}\label{ltl0.9}
 Let $u_0:\R^n \to \R$ be as in Theorem \ref{entire}. Then there exists a sequence of smooth functions $u^k_0:\R^n\to \R$  such that
 \begin{enumerate}
\item $u^k_0 \to u_0$ in $C^{1+\alpha}(B_R(0))$ for any $R$ and $\alpha \in (0, 1)$,
\item $-(1-\delta)I_n \leq D^2 u^k_0 \leq (1-\delta)I_n$ for every $k$,
 \item $\sup_{x\in\R^n}|D^l u^k_0| \leq C_{l, k}$ for every $l \geq 2$ and $k$.
 \end{enumerate}
 \end{lem}
 \begin{proof}
 Let $$u^k_0(x)=\int_{\R^n} u_0(y)K(x, y, 1/k)dy$$
where $K(x, y, t)$ is the standard heat kernel on $\R^n\times (0, \infty)$.    Conditions (1) and (2) and smoothness of $u^k_0$ are easily verified.  Now note that by hypothesis, $D_y^2 u_0(y)$ is a well defined and uniformly bounded function almost everywhere on $\R^n$ and that we may write 
$$
D_x^lu^k_0 (x)=\int_{\R^n} D_y^2 u_0(y)D_x^{l-2}K(x, y, 1/k)dy
$$ 
for every $l\geq 1$ from which it is easy to see that condition (3) is also true.  \end{proof}

Now consider the nonparametric mean curvature flow equation for  $f:\R^n \to \R^m$ 
 \begin{equation}\label{GMCF}
\left\{%
\begin{array}{ll}
 & \dfrac{df^{a}}{dt} =  g^{ij}(f)( f^{a})_{ij}\\
 &f(x, 0)=f_0(x)
\end{array}%
\right.
\end{equation}
where $g^{ij}(f)$ is the inverse of $g_{ij}(f):= \delta_{ij}+\sum_a f^a_i f^a_j$.  
 
 \begin{prop}\label{ap1}
 Suppose $f_0:\R^n \to \R^m$ is a smooth function such that for each $l\geq 1$, we have $\sup |D^l f_0| \leq C_{l}$ for some constant $C_l$.  Then \eqref{GMCF} has a shorttime smooth solution $f(x, t)$ with initial condition $f_0$ such that $\sup|D^l f|<\infty$ for every $l$ and $t$.
 \end{prop}

 \begin{proof}
 Consider the following quasilinear parabolic system for $v:\R^n \to \R^m$
 
  \begin{equation}\label{GMCF2}
\left\{%
\begin{array}{ll}
 & \dfrac{dv^{a}}{dt} =  g^{ij}(f_0+v)( v^{a})_{ij}+ g^{ij}(f_0+v)(f_0^{a})_{ij}\\\
 &v(x, 0)=0
\end{array}%
\right.
\end{equation}
where $g^{ij}(f_0+v)$ is the inverse of $g_{ij}(f_0+v)=\delta_{ij}+\sum_a (f_0+v)^a_i (f_0+v)^a_j$.  We begin by proving \eqref{GMCF2} has a shorttime smooth solution.  The proof is based on a general implicit function theorem argument explained in \cite{H1}  (see Theorem 5.1 in \cite{H1}).

For any $k$ let $C^{k+\alpha, k/2+\alpha/2}$ be the Banach space of vector valued functions $v:\R^n \times [0, 1]\to \R^m$ which are componentwise in $C^{k+\alpha, k/2+\alpha/2}(\R^n\times[0, 1])$, where the norm on these spaces is  defined by \[\|v\|_{ C^{k+\alpha, k/2+\alpha/2}}=\max_{1\leq a\leq m} \|v^a\|_{C^{k+\alpha, k/2+\alpha/2}(\R^n\times[0, 1])}.\]   Define
\begin{equation}
 \begin{split}
 B_1&:=\{ v \in C^{2+\alpha, 1+\alpha/2}: v(x, 0)=0\}\\
  B_2&:=\{ w\in C^{\alpha, \alpha/2}\}\\
\end{split}
\end{equation}
And define the map $\Phi:B_1\to B_2$ by
\begin{equation}
[\Phi(v)]^a:=\dfrac{dv^{a}}{dt}-g^{ij}(f_0+v)( v^{a})_{ij}- g^{ij}(f_0+v)(f_0^{a})_{ij}.
\end{equation}

Then the linearization $D\Phi_{v_0}$ of $\Phi$ at any $v_0\in B_1$ corresponds to a uniformly parabolic linear system on $\R^n$ with coefficients in $C^{\alpha, \alpha/2}$  and zero initial condition, and it follows from standard linear parabolic theory that $D\Phi_{v_0}$ is an isomorphism from $T_{v_0}B_1$ to $T_{w_0}B_2$ where $w_0=\Phi(v_0)$.   Thus by the implicit function theorem, there exists $\epsilon>0$ such that $\Phi(v)=w$ has a solution $v\in B_1$ provided $\| w-w_0\|_{C^{\alpha,\alpha/2}} \leq \epsilon$. 

Now let
 \[
\begin{split}F_1(x):= &g^{ij}(f_0) (f_0)_{ij}, \\ F_2(x):= &g^{ij}(f_0) (F_1)_{ij}-g^{pi}(f_0)g^{qj}(f_0)\left[\p_p(f_0)\p_qF_1+\p_q(f_0)\p_pF_1\right](f_0)_{ij}.\end{split}
\] 
Then $F_1$ and $F_2$ are smooth functions belonging to $C^{k+\alpha, k/2+\alpha/2}$ for every $k$.  Set 
\[v_0(x, t):=tF_1(x)+\dfrac{t^2}{2} F_2(x),\]
 then $v_0\in B_1$.  It follows that
 \begin{equation}w_0=\Phi(v_0)=F_1+tF_2-g^{ij}\left(f_0+tF_1+\dfrac{t^2}{2} F_2(x)\right)\left(f_0+tF_1+\dfrac{t^2}{2} F_2(x)\right)_{ij}
\end{equation} 
and $w_0\in C^{k+\alpha, k/2+\alpha/2}$ for every $k$ and $w_0(x, 0)=0$.
We  compute 
\begin{equation}\label{2ndorder}
\begin{split}
\frac{\partial }{\partial t}w_0(x, 0)=&F_2-g^{ij}\left(f_0\right) (F_1)_{ij}\\
&+g^{pi}(f_0)g^{qj}(f_0)\left[\p_p(f_0)\p_qF_1+\p_q(f_0)\p_pF_1\right](f_0)_{ij}\\
=&0.
\end{split}
\end{equation}

Now for any $\tau \in (0, 1)$ define the function 
 \begin{equation}
w_\tau(x, t)=\left\{%
\begin{array}{ll}
 &   0, \hspace{2.1cm} t\leq \tau\\
 & w_0(x, t-\tau),\hspace{0.25cm} t>\tau
\end{array}%
\right.
\end{equation}
 Then by \eqref{2ndorder} we see that $w_{\tau} \in C^{k, 1}$ for every $k$, in particular $|D_xD_tw_\tau|\leq C$ for some constant $C$ and $w_\tau\in B_2$.   Moreover for $\tau>0$ sufficiently small we have $\| w_{\tau}-w_0\|_{C^{\alpha,\alpha/2}} \leq \epsilon$.  Thus for such a $\tau$ we have $\Phi(v)=w_{\tau}$ for some $v\in B_1$ as explained above.  But this gives $\Phi(v)=0$ for  $t\in [0, \tau)$ by the definition of $w_{\tau}$, which in turn implies $v(x, t)$ solves \eqref{GMCF2} for $t\in [0, \tau)$.  Then the smoothness of $v(x, t)$ follows from a standard bootstrapping argument applied to \eqref{GMCF2}. 
It is now easy to see that $f(x, t):=v(x, t)+f_0(x)$ provides a smooth shorttime solution  to \eqref{GMCF} with $\sup|D^lf|<\infty$ for each $l$ and $t$.\end{proof}

\begin{lem}\label{ltl1}
Let $\{ u^k_0(x)\}$ be the sequence in Lemma \ref{ltl0.9}.  Then for each $k$, \eqref{PMA} has a smooth solution $u^k(x, t)$ on $\R^n\times [0, \infty)$ with initial condition $u^k_0(x)$ such that:
 \begin{enumerate}
 
 \item $-(1-\delta)I_n\leq D^2u^k\leq (1-\delta)I_n$,  
 \item For any compact subset $S\subset \R^n \times [0, \infty)$ we have
  $$\|u^k\|_{C^{1+\alpha, \beta}(S)} \leq C_{S},$$ where $C_{S}$ is a constant depending only on $S$.  

 \item For any $l, m$ and compact subset $K\subset \R^n\times (0, \infty)$ we have
$$
 \|u^k(x, t)\|_{C^{l, m}(K)} \leq C_{l, m, \delta, K}
$$
where $C_{l, m, \delta, K}$ is a constant depending only on $l, m, \delta$ and $K$.  
  
 \item $\sup_{x\in \R^n}|D^l u^k(x,t)|^2 < C_{l,\delta} /t^{l-2} $  for all $l\geq3$ and some constant $C_{l, \delta}$ depending only on $l$ and $\delta$.
\end{enumerate}
\end{lem}
\begin{proof} Fix some $k$.  By taking $m=n$ in Proposition \ref{ap1} we see that  \eqref{GMCF} has a shorttime smooth solution with initial condition $f_0=Du^k_0$.  It follows that \eqref{LMCF} has a smooth shorttime solution with initial condition $F(x, 0)=(x, f_0(x))$ (see for example p.19 in \cite{Br}), and thus by Lemma 2.1 that \eqref{PMA} also has a shorttime solution $u^k(x, t)$ on $\R^n\times[0, T_k)$ for some $T_k>0$ with initial condition $u^k_0$.  We can assume further that $T_k$ is the largest such time. Then by condition (2) in Lemma \ref{ltl0.9} and Lemma \ref{L-4-1}, we have $-(1-\delta)I_n\leq D^2 u^k\leq(1-\delta)I_n$ in $[0,T_k)$, and thus $|D^2 u^k|\leq n$ in $[0,T_k)$. 
  
Now we want to show $T_k=\infty.$ Suppose $T_k<\infty$.  By Proposition \ref{ap1}, $u^k(x, t)$ has bounded geometry for each $t\in [0, T_k)$, and thus so does $\tilde{F}^k(x, t)=(x, Du^k(x, t))$.  By Lemma 
\ref{pl1} it follows that the second fundamental form, and all its covariant derivatives, of the corresponding solution $F(x, t)$ to \eqref{LMCF} are uniformly bounded for each $t$.   It follows from the equation
 $$\frac{d}{dt}\tg = -2H^l h_{lij}$$
for \eqref{LMCF}  that $g_{ij}$ is equivalent to the Euclidean metric on $\R^n$ for each $t$.  Thus $F(x, t)$ has bounded geometry for each $t$, and by Corollary \ref{apec1}, $|H|^2\leq Ct^{-1}$ in $[0,T_k)$.  It follows that $|H|^2$ is bounded by a positive constant $C_k$ in $[T_k/2,T_k)$.  Thus by Lemma \ref{pmaapel3},  $|D^l u^k|\leq C_{l,k}$ in $[T_k/2,T_k)$ for all integer $l\geq 2$.  It then follows from Proposition \ref{ap1} that we can extend the solution beyond $T_k$. This contradicts the definition of $T_k$ and so $T_k=\infty$.  

We now establish the estimates in (2) and (3).  Since $|D^2u^k|\leq n$, it follows that $|\partial u^k/\partial t|\leq C(n)$ for some constant depending only on $n$.   Then as $u^k(0, 0)$ and $Du^k(0, 0)$ are uniformly bounded for all $k$, we conclude that for any compact set $S\subset \R^n\times [0, \infty)$ there exists a constant $C_S$ depending on $S$ such that
 \[\|u^k\|_{C^{1+\alpha, \beta}(S)} \leq C_{S}.\]  Now by Corollary \ref{apec1}, $|H|^2\leq Ct^{-1}$ for some $C$ depending only on $\delta$.   Thus by Lemma \ref{pmaapel3}, for any $T>0$ we have that $|D^lu^k|\leq C_{\delta, T}$ on $\R^n \times [T, \infty)$ for all $l\geq 3$. Consequently $|\partial^m u^k/\partial ^m t|$ is likewise uniformly bounded on $\R^n \times [T, \infty)$ by a constant depending only on $m$ and $C_{\delta, T}$ for $m\geq 0$.  Then for any compact set $K\subset \R^n\times(0, \infty)$, there exists some constant $C_{l, m, \delta, K}$ depending on $l, m, \delta, K$ such that 
\[
 \|u^k(x, t)\|_{C^{l, m}(K)} \leq C_{l, m, \delta, K}.
\]
  
Finally,  we show the decay estimates (4)  holds.  We will first prove it for $l=3$. This amounts to show  that $\sup_{\R^n}|D^3 u^k|^2t\leq C$ for $t\rightarrow \infty$ where $C$ is independent of $k$.  Assume this is not the case. Since $\sup_{\R^n} |D^3 u^k|$ is finite
for each $t$ as $u^k$ satisfies the bounded geometry condition, there exist $t_k >0$ such that
$$
\sup_{x\in\R^n, t\leq t_k} |D^3 u^k(x,t)|^2 t_k = \sup_{x\in\R^n} |D^3  
u^k(x,t_k)|^2 t_k := 2\mu_k \to \infty
$$
as $k\to\infty$. For each $k$,  there exists  
$x_k\in\R^n$ such that
$$
|D^3 u^k(x_k,t_k)|^2 t_k \geq \mu_k.
$$
 Let $(y, Du_{\lambda_k}(y, s))$ the parabolic scaling of the graph $(x, Du^k(x, t))$ for $t\leq t_k$ by $\lambda_k=\sqrt{\mu_k/t_k}$ at $(x_k, t_k)$ for each $k$.
By the same argument as in the proof of  Lemma \ref{pmaapel3}, for each $k$, $ u_{\lambda_k}(y, s)$ is still a solution of \eqref{PMA} for $s\in [-\mu_k, 0]$ with 
$$
D^2_{y}  u_{\lambda_k}=D^2_x u^k, \hspace{12pt} |D^3_{ y}  u_{\lambda_k}|\leq \sqrt{2}
$$ 
on $\R^n \times  [-\mu_k, 0]$, hence for any fixed $\mu<\mu_k$, all higher derivatives of $u_{\lambda_k}$ are uniformly bounded on
$\R^n \times  [-\mu, 0]$ for all sufficiently large $k$, by positive constants independent of $k$ by Lemma \ref{pmaapel4}.  Note that  $u_{\lambda_k}(0, 0)=0$ and $D_{y} u_{\lambda_k}(0, 0)=0$, it then follows that $ u_{\lambda_k}(y, s)$ converges subsequentially on each compact set on ${\mathbb R}^n\times (-\infty,0]$ to a solution $ u_\infty$ of \eqref{PMA} such that
\[
 -(1-\delta)I_n \leq D^2_{y} u_\infty \leq(1-\delta)I_n
 \]
 and
 \[
 |D^3_{y} u_\infty (0, 0)|\geq1. 
\]
Note that for each $u^k(x, t)$ we have $|H|^2t\leq C$ for all $t$ and some $C$ independent of $k$. Thus after scaling, for any fixed $\mu<\mu_k$ we have $|H_{\lambda_k}|^2\leq C\mu_k^{-1}\rightarrow 0$ uniformly on $\R^n \times [-\mu, 0]$  for all sufficiently large $k$. It follows that $H_\infty=0$ and so $(y, D_{y} u_\infty)$ is a special Lagrangian graph. But this contradicts the Bernstein theorem  for special Lagrangian graphs in \cite{Y}. So $|D^3u^k|^2 t\leq C$ for any $k$ and $t$.  

For $l\geq 4$  we use a blowdown argument. This follows essentially from Lemma \ref{pmaapel4} as we have shown that $|D^2u^k|$ and $ |D^3u^k|$ are bounded. Suppose that $|D^lu^k|t^{l-2}$ were not bounded.  Then there would exist $(x_k, t_k)$
such that $|D^lu^k|^2(x_k, t_k) t_k^{l-2}\geq \mu_k\rightarrow \infty$. Similarly we can pick $t_k$  such that $\sup_{t\leq t_k, x\in \R^n} |D^lu^k|^2 t^{l-2} \leq 2\mu_k$. 
Now let $\lambda_k=(t_k^{l-2}/\mu_k)^{\frac{1}{2(2-l)}}$ and let $(y, Du_{\lambda_k}(y, s)$  be the parabolic scaling of $(x, Du^k(x, t))$ for $t_k/2\leq t\leq t_k$ by $\lambda_k$ at $(x_k, t_k)$ for each $k$. 
After scaling,  $(y, D u_{\lambda}(y, s))$ is a solution of \eqref{PMA} for $s\in [-1/2, 0]$, and we have
\[
D^2_{y} u_{\lambda_k}=D^2_x u^k
\]
and
\begin{equation}\nonumber
\begin{split}
 |D^3_{y} u_{\lambda_k}( \cdot,s)|^2=&\lambda_k^{-2} |D^3_xu^k(\cdot, t)|^2\\=&\mu_k^{\frac{1}{2-l}}\left(|D^3u^k( \cdot,t)|^2t_k \right)\\
 \leq & 2\mu_k^{\frac{1}{2-l}}\left(|D^3u^k( \cdot,t)|^2t \right)\rightarrow 0 
 \end{split}
\end{equation}
as $k\to\infty$.
By Lemma \ref{pmaapel4}, all the higher derivatives of $u_{\lambda_k}$ are uniformly bounded.  Then as before, with $ u_{\lambda_k}(0, 0)=D_{y} u_{\lambda_k}(0, 0)=0$ we conclude that $u_{\lambda_k}$ converges subsequentially on compact sets to a solution $u_\infty$ of \eqref{PMA}. 
But again, $|D^3 u_\infty|=0$ everywhere contradicts $|D^l u_\infty (0, 0)|\geq1$. \end{proof}

\section{Longtime existence and convergence}

We are now ready to prove Theorems \ref{entire} and \ref{compact}.    

\begin{proof} [Proof of Theorem \ref{entire}]
Let $u_0(x)$ be as in Theorem \ref{entire}.  Let $\{u^k_0\}$ be a sequence of approximation of $u_0$ as in Lemma \ref{ltl0.9}. By Lemma \ref{ltl1}, for each $k$ we have a longtime solution $u^k(x, t)$ to \eqref{PMA}.

For any fixed positive $R,T,\epsilon$, by (2) and (3) in Lemma  \ref{ltl1}, there exists a subsequence of $\{u^k\}$, which we still denote by $\{u^k\}$, such that for some $u_{R, T}$ on $\overline{B_R(0)}\times[0,T]$ we have:
\begin{enumerate}
\item 
$
u^k\rightarrow u_{R,T}
$
in $C^{1+\alpha,\beta}(\overline{B_R(0)}\times[0,T])$ for any $\alpha,\beta\in(0,1)$ with
$$
\left| u_{R,T}\right|_{C^{1+\alpha,\beta}}\leq C(R,T)
$$
\item
$
u^k \rightarrow  (u_{R,T})|_{[\epsilon,T]} 
$
in $C^{l,m}(\overline{B_R(0)}\times [\epsilon,T])$ for any $l,m$ with
$$
\left| u_{R,T,\epsilon}\right|_{C^{l,m}}\leq C(l,m,R,T,\epsilon)
$$
\end{enumerate}
Then letting $R\to\infty,T\to\infty,\epsilon\to 0$ and using a diagonal subsequence argument, we may conclude that $\{u^k\}$ has a convergent subsequence converging 
on every compact subset of ${\mathbb R}^n\times[0,\infty)$ to a solution $u$ of (\ref{PMA}) which is smooth on ${\mathbb R}^n\times(0,\infty)$ and $\lim_{t\to 0} u(x,t)=u_0(x)$ for all $x\in{\mathbb R}^n$. In particular, we know that $u(x, t)\rightarrow u_0(x)$ in $C^{1+\alpha}(S)$ for each compact set $S\subset \R^n$ when $t\rightarrow 0.$ It is then not hard to see from Lemma 5.2 that $u(x, t)$ satisfies conditions (1)-(3) of Theorem \ref{entire}.  

As $t\rightarrow \infty$, we know that $|D^3u|$ (or the second fundamental form) decays at the rate $C/\sqrt{t}$. But in general, $u(x, t)$ may not converge as $t\rightarrow \infty$ since the graph $(x, Du(x, t))$ could move on the order $1/\sqrt{t}$. See \cite{HE} for the discussion for codimension one case.  If we assume in addtion $|Du_0(x)|\rightarrow 0$ for $|x|\rightarrow \infty$, then by Lemma \ref{pmaapel2} we know that $|Du(x, t)|$ stays bounded. 
Consider the equation for $u_k$
\[
\frac{\partial u_k}{\partial t}=\tilde g^{ij} u_{kij}
\]
As $|D^2u|$ is uniformly bounded thus so is $\tilde g^{ij}$ and the above equation is uniformly parabolic. We can then use the theorem  in \cite{I} to obtain that $Du(x, t)\rightarrow 0$ when $t\rightarrow 0$. 
\end{proof}

\begin{proof} [Proof of Theorem \ref{compact}]
 The convergence in the theorem follows from the estimates in Theorem \ref{entire} and the results in \cite{Si}.
\end{proof}
\section{Bernstein type theorem for translating solitons}
 We now prove Theorem \ref{bernstein}.  
 
 \begin{proof}
 Recall that an entire Lagrangian translating soliton graph $(x, Du(x))$ satisfies 

\begin{equation}\label{E-soliton-1.2}
\sum_i\arctan \lambda_i+\sum_i a_i\frac{\partial u}{\partial x_i}-\sum_ib_ix_i=c.
\end{equation}

\noindent{\bf Step 1.} If $|D^3u|$ is uniformly bounded, then $|D^lu|$ is uniformly bounded for all $l\geq 4$.

To see this, differentiating  \eqref{E-soliton-1.2} we get 
\begin{equation}\label{E-soliton-2}
g^{ij}u_{ijk}=-a_iu_{ik}+b_k,
\end{equation}
where $g_{ij}=\delta_{ij}+u_{ik}u_{kj}$ is the induced metric by $(x, Du(x))$.  Recall that $|D^2u|\leq n$.  A standard bootstrapping argument shows that if $|D^2u|, |D^3u|$ are bounded, then $|D^lu|$ is bounded for all $l\geq 4$. 

\vspace{15pt}
\noindent{\bf Step 2.}  $|D^3u|$ is uniformly bounded.

We first observe that $|H|$ is uniformly bounded. To see this, note that $|D^2u|\leq n$ and thus $g_{ij}$ is equivalent to the standard Euclidean metric.  Also $H\equiv T^{\bot}$ and $T$ is a constant vector from which it follows that $|H|$ is bounded.

 Suppose that $|D^3u|$ is not bounded.  Then as $u$ is smooth, there exists $x_k\in \mathbb{R}^n$ such that $|D^3u(x_k)|\rightarrow \infty$ while $|x_k|\rightarrow \infty$. 

{\bf Claim}  There exists a sequence $\{y_k\}$ such that for each $k$: $y_k\in B_{x_k}(2)$ and  for some nonnegative integer $n_k$,  $|D^3u(y_k)|=2^{n_k}|D^3u(x_k)|$ and $|D^3u(y)|\leq 2|D^3u(y_k)|$ for any $y\in B_{y_k}(2^{-n_k})$.

 We establish the claim by using Perelman's point selection technique \cite{P}.  Fix any $x_k$.  If  $|D^3u(x)|\leq 2|D^3u(x_k)|$ for all $x\in B_{1}(x_k)$ then the claim holds for $x_k$.  If this is not the case, however, there must exist $x^1_k\in B_{1}(x_k)$ such that $|D^3u( x^1_k)|=2|D^3u(x_k)|$ by the intermediate value theorem.   Suppose that $x^l_k$ is constructed as above for some $l\geq 1$ while $|D^3u(x)|\leq 2|D^3u(x^l_k)|$ is not satisfied for some $x\in B_{2^{-l}}(x^l_k)$.  Then as before there exists $x^{l+1}_k\in B_{2^{-l}}(x^l_k)$ such that  $|D^3u(x^{l+1}_k)|=2|D^3u(x^l_k)|.$ Clearly, the sequence $\{x^l_k\}$ satisfies $|D^3u(x^l_k)|=2^l|D^3u(x_k)|$ and $d(x^l_k, x_k)\leq \sum_{i=0}^{l} 1/2^i\leq 2$ for each $l$.   Since $u$ is smooth and $|D^3u(x_k)| \geq 1$ for $k$ sufficiently large, this process can generate at most a finite number of points $x_k^1, x_k^2,...,x_k^{n_k}$.  The claim is established by letting $y_k=x_k^{n_k}$ for each $k$.

 For each $k$, in the notation above we have $y_k\in B_{x_k}(2)$ and  $|D^3u(y_k)|=2^{n_k}|D^3u(x_k)|$ and $|D^3u(y)|\leq 2|D^3u(y_k)|$ for all $x\in B_{y_k}(2^{-n_k})$. Denote $|D^3u(x_k)|=\mu_k$ and $\lambda_k=|D^3u(y_k)|=2^{n_k}\mu_k$.   We consider a sequence of scalings of $(x, Du(x))$ as follows,
 \begin{equation}
 \begin{split}
 z&=\lambda_k(x-y_k)\\
 u_{\lambda_k}(z)&=\lambda_k^2 [u(x)-u(y_k)-Du(y_k) \cdot (x-y_k)]
\end{split}
 \end{equation}
 After scaling, $u_{\lambda_k}(0)=D_zu_{\lambda_k}(0)=0,$ and for any integer $l\geq 1$,
 \[
 D^l_zu_{\lambda_k}=\lambda^{2-l}_k D^l_xu.
 \]
Moreover  the graph $(x, Du)$ over $B_{y_k}(2^{-n_k})$ becomes the graph $(z, D_zu_{\lambda_k})$ over $B_0(\mu_k)$. Also we have $H_{\lambda_k}=\lambda_k^{-1}H$ and $|D^3_zu_{\lambda_k}|\leq 2$ for all $z\in B_0(\mu_k)$.  By \eqref{E-soliton-2}, it is clear that
 \begin{equation}\label{E-soliton-3}
 (g_{\lambda_k})^{ij} (u_{\lambda_k})_{ijk}=\lambda_k^{-1}(a_i (u_{\lambda_k})_{ik}-b_i),
 \end{equation}
 where $(g_{\lambda_k})_{ij}=\delta_{ij}+(u_{\lambda_k})_{ik}(u_{\lambda_k})_{kj}.$
By a standard bootstrapping argument, for any $z\in B_0(\mu_k-1)$, all higher derivatives of $u_{\lambda_k}(z)$ are uniformly bounded since $|D^3_zu_{\lambda_k}|\leq 2$.  This combined with $u_{\lambda_k}(0)=D_zu_{\lambda_k}(0)=0$ give a bound on $\|u_{\lambda_k}\|_{C^l}$ for each $l$ which is independent of $k$ and uniform for each compact set in $\R^n$ since $\mu_k\rightarrow \infty$.   Then by the Arzela-Ascoli theorem $u_{\lambda_k}$ converges subsequentially on any compact subset of $\mathbb{R}^n$ to a smooth function $u_\infty$ on $\mathbb{R}^n$.  Since $H_{\lambda_k}=\lambda_k^{-1}H$ and $H$ is uniformly bounded for $(x, Du)$, then $|H_\infty|\equiv 0$ for $(z, D_zu_\infty)$, and so $(z, D_zu_\infty)$ is a special Lagrangian graph. We know that
\[
-I_n\leq D^2u_\infty\leq I_n.
\]
 But  the Bernstein theorem in \cite{Y} contradicts $|D^3u_\infty(0)|=1$.  The completes Step 2. 

\vspace{15pt}
\noindent{\bf Step 3.} The second fundamental form $A\equiv0.$

 Let $F(x, t)$ be the translating solution to \eqref{LMCF} generated by $(x, Du_0(x))$, and let $u(x, t)$ be the corresponding solution to $\eqref{PMA}$ with initial condition $u_0$.   Then as $F(x, 0)$ has bounded geometry by Steps 1 and 2, it follows that $F(x,t)$ does as well for each $t$, and in particular $\sup_{x\in\R^n}|A(x, t)|$ is constant in $t$.  On the other hand, by the proof of Lemma 5.2, $u(x, t)$ must then satisfy $$\sup_{x\in \R^n} |D^3 u(x, t)|^2 \leq C/t$$ for all $t$ and some constant $C$ depending only on $\delta$.  It follows that  $\sup_{x\in\R^n}|A(x, t)|=0$ for all $t$, and thus $(x, Du_0(x))$ must be a flat plane. \end{proof}

  \bibliographystyle{amsplain}

\end{document}